\documentclass[11pt]{amsart}

\title{Improved Bounds for Radial Projections in the Plane}
\author{Marianna Cs\"ornyei} 
\address{Department of Mathematics, University of Chicago, Chicago, IL 60637}  
\email{csornyei@math.uchicago.edu}

\author{D. M. Stull}
\address{Department of Mathematics, University of Chicago, Chicago, IL 60637}
\email{dmstull@uchicago.edu}

\usepackage{amsmath,amssymb,amsthm,amsfonts,mathtools,bm,xcolor}

\theoremstyle{plain}
\newtheorem{theorem}{Theorem}[section]

\newtheorem*{claim*}{Claim}

\newtheorem{proposition}[theorem]{Proposition}
\newtheorem*{proposition*}{Proposition}

\newtheorem*{fact*}{Fact}

\newtheorem*{conjecture*}{Conjecture}

\newtheorem{lemma}[theorem]{Lemma}
\newtheorem*{lemma*}{Lemma}

\newtheorem*{question*}{Question}
\theoremstyle{definition}\newtheorem{remark}[theorem]{Remark}
\theoremstyle{definition}\newtheorem*{remark*}{Remark}
\theoremstyle{definition}\newtheorem{definition}[theorem]{Definition}
\theoremstyle{definition}\newtheorem*{definition*}{Definition}
\theoremstyle{definition}
\theoremstyle{definition}
\theoremstyle{definition}
\theoremstyle{definition}\newtheorem*{example*}{Example}

\renewcommand{\phi}{\varphi} 
\renewcommand{\epsilon}{\varepsilon} 
\def\eps{{\epsilon}}
\def\emp{{\emptyset}}

\def\N{\mathbb{N}} 
\def\R{\mathbb{R}} 
\def\P{\mathbb{P}} 

\def\S{\mathcal{S}}
\def\P{\mathcal{P}}
\def\T{\mathcal{T}}
\def\D{\mathcal{D}}
\begin{document}

\begin{abstract}
    We improve the best known lower bound for the dimension of radial projections of sets in the plane. We show that if $X,Y$ are Borel sets in $\R^2$, $X$ and $Y$ are not contained in the same line and $\dim_H(X)>0$, then
    $$\sup\limits_{x\in X} \dim_H(\pi_x Y) \geq \min\left\{(\dim_H(Y) + \dim_H(X))/2, \dim_H(Y), 1\right\},$$ 
where $\pi_x Y$ is the radial projection of the set $Y$ from the point $x$.
\end{abstract}

\maketitle
\section{Introduction}
Given a point $x\in\R^d$, the radial projection with respect to $x$ is the function $\pi_x : \R^d \setminus \{x\} \to \S^1$ defined by
\begin{center}
    $\pi_x y = \frac{y-x}{|x-y|}$.
\end{center}

Recently, there has been substantial progress in our understanding of how the Hausdorff dimension of a set is changed when projected radially onto a point \cite{Orponen19DimSmooth, OrpShmWan24, Ren23}. In particular, we focus on the following problem. Let $X, Y\subseteq\R^d$ be Borel sets. We would like to give sharp lower bounds on the quantity
\begin{center}
    $\sup\limits_{x\in X} \dim_H(\pi_x (Y\setminus \{x\}))$
\end{center}
in terms of the Hausdorff dimensions of $X$ and $Y$.

In this paper we will be concerned with radial projections in $\R^2$. There has also been exciting progress on this question in higher dimensions, e.g., \cite{Ren23, BriFuRen24}.

Orponen, Shmerkin and Wang \cite{OrpShmWan24} proved the following theorem on the dimension of radial projections in the plane.
\begin{theorem}\label{thm:OSW1}
    Let $X\subseteq\R^2$ be a Borel set which is not contained on any line. Then, for every Borel set $Y\subseteq\R^2$,
    \begin{center}
    $\sup\limits_{x\in X} \dim_H(\pi_x (Y\setminus \{x\})) \geq \min\{\dim_H(X), \dim_H(Y), 1\}$.
\end{center}
\end{theorem}
This theorem is sharp whenever $\dim_H(Y) = \dim_H(X)$. In the same paper \cite{OrpShmWan24}, they also proved the following theorem.
\begin{theorem}\label{thm:OSW2}
    Let $X\subseteq\R^2$ be a non-empty Borel set. Then, for every Borel set $Y\subseteq\R^2$, such that $\dim_H(Y) > 1$,
    \begin{center}
    $\sup\limits_{x\in X} \dim_H(\pi_x (Y\setminus \{x\})) \geq \min\{\dim_H(X)+\dim_H(Y)-1, 1\}$.
\end{center}
\end{theorem}

In this paper, we improve their lower bounds, for the case where $\dim_H(X) > 0$. We prove that:
\begin{theorem}\label{thm:mainthm1}
    Let $X\subseteq\R^2$ be a Borel set such that $\dim_H(X) > 0$, and let $Y\subseteq\R^2$ be a Borel set. If $X$ is not contained on a line, then 
    \begin{equation}\label{our1}
        \sup\limits_{x\in X} \dim_H(\pi_x (Y\setminus \{x\})) \geq \min\left\{(\dim_H(Y) + \dim_H(X))/2, \dim_H(Y), 1\right\}.
    \end{equation}
    If $X \subseteq \ell$, for some line $\ell \subseteq\R^2$, then 
    \begin{equation}\label{our2}
        \sup\limits_{x\in X} \dim_H(\pi_x (Y\setminus \{x\})) \geq \min\left\{(\dim_H(Y\setminus \ell) + \dim_H(X))/2, \dim_H(Y\setminus \ell), 1\right\}.
    \end{equation}
\end{theorem}

We note that Theorem \ref{thm:mainthm1} also improves the result of Ren and Wang giving sharp bounds on the exceptional sets for orthogonal projections in $\R^2$. They proved the following in \cite{RenWan23}.
\begin{theorem}\label{ort}
    Let $Y\subseteq\R^2$ be a Borel set. Then, for all $0\leq u \leq \min\{\dim_H(Y), 1\}$, we have
    \begin{center}
        $\dim_H(\{\theta \in \mathcal{S}^1\mid \dim_H(p_\theta Y) < u\}) \leq \max\{2u - \dim_H(Y), 0\}$,
    \end{center}
    where $p_\theta$ is the orthogonal projection map.
\end{theorem}
Equivalently, if $X\subseteq \ell$, for some line $\ell \in\R^2$, $Y$ is disjoint from $\ell$, and $\dim_H(X) > \max(2u-\dim_H(Y),0)$ for $0 \leq u \leq \min\{\dim_H(Y), 1\}$, then $$\sup_{x\in X} \dim_H(\pi_x Y) \geq u.$$ It is immediate to check that our Theorem \ref{thm:mainthm1} implies the same result, moreover, it implies Theorem \ref{ort} \emph{without using the additional assumption that $X$ is contained in a line}.

\begin{remark}An elementary calculation also shows that our lower estimate is always at least as large as the lower estimate in Theorem \ref{thm:OSW1}, and it is strictly larger when $\dim_H(X)< \min\{\dim_H(Y),1\}$. Similarly, the estimate in \eqref{our1} is always at least as large as the estimate in Theorem \ref{thm:OSW2}, and it is strictly larger in the case when in Theorem \ref{thm:OSW2} we have $\dim_H(X)+\dim_H(Y)-1<1$.
\end{remark}

\begin{remark}
    We will also state a discretized analog of Theorem \ref{thm:mainthm1}, whose proof is similar to the continuous proof. See Theorem \ref{thm:mainTheoremDiscretized}.
\end{remark}

The main new ingredient of the proof of Theorem \ref{thm:mainthm1} is a surprising inequality bounding Kolmogorov complexity of a line, which may be of independent interest. In particular we show in Theorem \ref{thm:boundComplexityLine}, that, if $X,Y\subseteq\R^2$ are compact sets of positive dimension, then we can always find a suitable oracle $A$ and large effective dimensional points $x\in X$ and $y\in Y$ such that, for all sufficiently large $n$, $K_n^A(\ell)\le K_n^A(\ell\mid x)+K_n^A(\ell\mid y)+O(\log n)$, where $\ell$ is the line containing $x$ and $y$. (For a more precise statement see Theorem \ref{thm:boundComplexityLine}.)
The proof of Theorem \ref{thm:mainthm1} follows by combining this result and the incident bound due to Ren and Wang \cite{RenWan23} which was used in their solution to the Furstenberg set problem. 

We conclude this introduction by noting that improvements on the dimension of radial projections is closely related to many problems in geometric measure theory. Progress on radial projections played a  crucial role in the work of Orponen and Shmerkin, and Ren and Wang on Furstenberg sets and projections \cite{OrpShm23, RenWan23}, as well as the work of Wang and Zahl on sticky Kakeya sets \cite{WanZah22}, which in turn was a crucial step in the solution of the Kakeya problem in $\R^3$ \cite{WanZah25}. 

Radial projections were also important in most of the recent progress on the distance set problem  \cite{Shmerkin20,KelShm19,Stull22c, DuOuRenZhang23a, FieStu23}. For example, in a forthcoming paper, along with Cholak, Lutz, Lutz and Mayordomo, we will use the results of this paper to improve the best known bound of pinned distance sets for planar sets whose dimension is at most one. In particular, the forthcoming paper proves the following theorem.
\begin{theorem}\label{thm:distanceBound}
    Let $E\subseteq\R^2$ be Borel and $s =\dim_H(E) \leq 1$. Then
    \begin{center}
        $\sup_{x\in E} \dim_H(\Delta_x E) \geq \frac{3s}{4}$,
    \end{center}
    where $\Delta_x y = |x - y|$ for $y\in \R^2$.
\end{theorem}
Previously, Shmerkin and Wang \cite{ShmWan25} had proved that 
\begin{equation}\label{eq:ShmWangBound}
    \sup\limits_{x\in E} \dim_H(\Delta_x E) \geq \frac{s}{2} + \frac{s^2}{2(2+\sqrt{s^2 +4})},
\end{equation}
for every Borel set $E\subseteq\R^2$ with $\dim_H(E) = s \leq 1$. Fiedler and Stull \cite{FieStu24} improved this bound, so long as $s$ was sufficiently large. They proved that
\begin{equation}
    \sup\limits_{x\in E}\dim_H(\Delta_x(E))\geq s \left(1 - \frac{2-s}{2(1+2s-s^2)}\right),
\end{equation}
which is larger than (\ref{eq:ShmWangBound}) so long as $s \geq .48$. We note that Theorem \ref{thm:distanceBound} gives an improved bound of both results for all $0 < s < 1$.

\section{Preliminaries}

\subsection{Discretized incidences}

We will use the standard point-line duality in the plane. 

Without loss of generality, we can assume in Theorem \ref{thm:mainthm1} that $X,Y\subset [0,1]^2$. Also, without loss of generality, it is enough to consider the directions of those lines joining points of $X$ and $Y$ whose dual is in $[0,1]^2$.

Let $\delta$ be a dyadic scale. We denote the set of dyadic cubes in $[0,1]^2$ of side length $\delta$ by $\D_\delta$. 
We call a set of lines a $\delta$-tube if it is the dual set of the points covered by a cube $Q\in\D_\delta$. We denote the set of $\delta$-tubes by $\T^\delta$.

We say that a $\delta$-cube $Q$ intersects a $\delta$-tube $T$, if $T$ contains a line that meets $Q$, and we denote this by $T\cap Q\neq \emptyset$.

In \cite{RenWan23} Ren and Wang proved sharp bounds on the incidences between $\delta$-tubes and $\delta$-cubes, leading to a resolution of the Furstenberg set problem. We will use their work to prove our main theorem. They introduced the following definitions, and proved Theorem \ref{thm:RenWang} below (see Theorem 4.1 in \cite{RenWan23}).

For any bounded set $A\subseteq\R^d$, and dyadic scale $\delta$, denote
\begin{center}
    $\D_\delta(A) = \{Q\in\D_\delta \mid Q \cap A \neq \emptyset\}$
\end{center}
and 
\begin{center}
    $\vert A \vert_\delta = \vert \D_\delta(A)\vert$.
\end{center}

Let $s\in[0,d]$ and $C > 0$. A non-empty bounded set $A\subseteq\R^d$ is called a $(\delta, s, C)$-set if it satisfies the following Frostman type condition: 
\begin{equation}
    \vert A \cap B(x,r)\vert_\delta \leq Cr^s \vert A \vert_\delta \;\;\; \forall x \in \R^d, \; r \in [\delta, 1].
\end{equation}

When $\P\subseteq \D_\delta$, we say that $\P$ is a $(\delta, s, C)$-set if the set of points in $\R^2$ covered by the cubes in $\P$ is a $(\delta, s, C)$-set.
A set of $\delta$-tubes $\T$ is called a $(\delta, s, C)$-set if its dual is a $(\delta, s, C)$-set.

\medskip

\begin{theorem}\label{thm:RenWang}
   For every $\epsilon > 0$, $s\in (0,1]$, and $t\in (0, 2]$ there exists $\eta > 0$ s.t. for every small enough $\delta>0$ and for every $M\in\N$ the following holds.
   
   Suppose that  $\mathcal{P}$ is a $(\delta, t, \delta^{-\eta})$-set, and suppose that for each $Q\in\P$ there is a $(\delta, s, \delta^{-\eta})$-set of tubes $T(Q)$ such that each tube in $T(Q)$ intersects $Q$, and $|T(Q)|\sim M$. Then for $\T:=\bigcup_{Q\in\P} T(Q)$ we have
   \begin{equation}
       \vert \mathcal{T}|\gtrsim(\log \eps^{-1})^{-c}\delta^{-\min\{t, \frac{s+t}{2}, 1\} + \epsilon} M
   \end{equation}
for some absolute constant $c$.
\end{theorem}

We will also use a result of Orponen, Shmerkin and Wang \cite{OrpShmWan24} on the existence of `thin tubes'. This was used in their paper establishing strong lower bounds on radial projections in $\R^2$. 

Recall that a probability measure $\mu$ is said to satisfy the \emph{$s$-dimensional Frostman condition}, if $\mu(B(x,r)) \leq Cr^s$ for some fixed constant $C > 0$, for all $x\in \R^d$ and all $r > 0$.

In \cite{OrpShmWan24} (Corollary 2.21) the authors proved the following statement:

\begin{theorem}\label{cor:OrpShmWan}
Let $s \in (0,1]$ and let $\mu, \nu$ be probability measures in the plane that satisfy the $s$-dimensional Frostman condition. 
Assume $\mu(\ell)\nu(\ell)<1$ for every line $\ell$. Then for every $0\leq \sigma < s$ there exists a Borel set $E\subset\R^2\times \R^2$ of positive $\mu\times\nu$ measure and a constant $C>0$ such that the following holds.

For every $x\in\R^2$, $r>0$ and for every rectangle $x\in R\subset\R^2$ of width $r$,
$$\nu(\{y\in R\mid (x,y)\in E\})\le Cr^\sigma.$$

\end{theorem}

\subsection{Algorithmic methods}
Our proof will use the algorithmic methods which have recently been applied to geometric measure theory \cite{AltBusWil25,BusFie25,Lutz21, LuQiYu24,LuLuMa23a}. We will assume basic knowledge about the fundamental definitions and theorems of Kolmogorov complexity, effective dimension and the use in geometric measure theory. For a more detailed introduction for mathematicians working in geometric measure theory who are unfamiliar with the algorithmic methods, the authors are in the process of writing lecture notes on this topic. These will not assume that the reader has any apriori background knowledge in computability theory. 

For any $x\in\R^d$ and $n\in\N$, the Kolmogorov complexity of $x$ at precision $n$ is
\begin{center}
    $K_n(x) = K(Q)$,
\end{center}
where $K(Q)$ is the Kolmogorov complexity of the dyadic cube with side length $2^{-n}$ containing $x$. We note that in the literature, $K_n(x)$ is more commonly defined using rational balls, however the definitions are equivalent up to an additive $O(\log n)$ factor, which can be safely ignored. 

We frequently use the following notation. For any $x\in\R^d$ and $n\in\N$, we will write
\begin{center}
    $K_n(x) \lessapprox L$
\end{center}
to mean that $K_n(x) \leq L + O(\log n)$. Similarly, we will write \begin{center}
    $K_n(x) \gtrapprox L$
\end{center}
to mean that $K_n(x) \geq L + O(\log n)$. We write $K_n(x) \approx L$ when $K_n(x) \lessapprox L$ and $K_n(x) \gtrapprox L$. More generally, when we write formulas involving several Kolmogorov complexities, possibly at different precisions (e.g. $K_{n_1}(x)\lessapprox K_{n_2}(y)-K_{n_3}(z)$), by $\lessapprox$ we mean that the inequality holds up to an error term $\sum_iO(\log n_i)$, where we are summing for all the precisions $n_i$ that appear (on either side of $\lessapprox$) in all the Kolmogorov complexities in our formula. 

One of the most important results in Kolmogorov complexity is the \textit{symmetry of information}. In \cite{LutStu20}, Lutz and Stull proved that a version of this holds in Euclidean space, namely
\begin{equation*}
    K^A_{n,m}(x, y) \approx K^A_{n, m}(x\vert y) + K^A_m(y)
\end{equation*}
for every oracle $A$, $x\in\R^{d_1}, y\in\R^{d_2}$ and precisions $n,m \in \N$.

In this paper, the main application of Kolmogorov complexity is to define a notion of local dimension of points in Euclidean space, first defined by Lutz \cite{Lutz03a}. We use the Kolmogorov characterization proved by Mayordomo \cite{Mayordomo02}. In particular, the effective Hausdorff dimension of $x\in\mathbb{R}^n$ relative to an oracle $A$ is
\begin{equation*}
    \dim_H^A(x) = \liminf \frac{K^A_n(x)}{n}.
\end{equation*}

We are able to use effective methods in geometric measure theory because of the following characterization of the Hausdorff dimension, proved by Lutz and Lutz \cite{LutLut18}.
\begin{theorem}\label{thm:pointtoset}
    For every $E\subseteq\mathbb{R}^n$,
    \begin{equation}\label{eq:HPtS}
        \dim_H(E) = \min_{A}\sup_{x\in E}\dim_H^A(x).
    \end{equation}

\end{theorem}
 We call an oracle $A$ a Hausdorff oracle for $E$ if it achieves the minimum in \eqref{eq:HPtS}.
 
 A set $K \subseteq \R^n$ is \textit{effectively compact relative to} $A$ if $K$ is compact and the set of finite open covers of $K$ by dyadic cubes is computably enumerable relative to $A$. We will use the fact that every compact set is effectively compact relative to some oracle. Hitchcock \cite{Hitchcock05} and J. Lutz \cite{Lutz03a} proved that, if $K$ is effectively compact relative to an oracle $A$, then $A$ is a Hausdorff oracle for $K$. For the purposes of this paper, we conclude that, if $Y\subseteq\R^2$ is computably compact relative to $A$ and $x\in\R^2$, then 
 \begin{equation}\label{don}
     \dim_H(\pi_x Y) = \sup\limits_{y\in Y} \dim^{A,x}_H(\pi_x y).
 \end{equation}
If $\mu$ is a Borel probability measure on $\R^d$, we say that $\mu$ is computable relative to an oracle $A$ if, for every dyadic cube $Q\in\D$, $\mu(D)$ is computable relative to $A$. We note that every Borel probability measure is computable relative to some oracle. If $K\subseteq\R^d$, $\mu$ is a probability measure whose support is contained in $K$ and $\mu$ is computable relative to an oracle $A$, then it was shown in \cite{Stull22c} that, for every oracle $B$, 
\begin{equation}\label{dd}
    K^{A,B}_n(x) \gtrapprox K^A_n(x)
    \end{equation}
for $\mu$-almost every $x\in K$.

We will frequently use the well-known fact relating Kolmogorov complexity and enumerable sets. Suppose that $\P \subseteq \D_{2^{-n}}$ is enumerable, given information encoded into a binary string $w$. If $x \in \R^2$ and $Q(x) \in \P$, where $Q(x)$ is the dyadic cube containing $x$, then
\begin{equation}
    K_n(x\mid w) \lessapprox \log \vert \P \vert.
\end{equation}
An immediate consequence is that 
\begin{equation}
    K_n(x) \lessapprox \log \vert \P \vert + \vert w \vert,
\end{equation}
where $\vert w \vert$ denotes the length of $w$.

\section{Complexity of lines}
We will need a slight modification of the definition of $(\delta, s, C)$-sets which is better suited for computability.

\begin{definition}
A set $A \subseteq\D_{2^{-n}}$ is called a \emph{dyadic $(2^{-n}, s, C)$-set} if for each $m\leq n$ and each $Q\in D_{2^{-m}}$, $Q$ has at most $C2^{-sm}\vert A \vert$ many subcubes in $A$.

\end{definition}

\begin{remark}\label{remark0}
Note that, if $A\subseteq\D_{2^{-n}}$ is a dyadic $(2^{-n}, s, C)$-set, then $A$ is a $(2^{-n}, s, C^\prime)$-set, where $C^\prime$ is a constant multiple of $C$. 
\end{remark}

\begin{remark}\label{remark}We also note that, if $A \subseteq\D_{2^{-n}}$, and $s$ and $C$ are rational, then determining if $A$ is a dyadic $(2^{-n}, s, C)$-set is computable, given $n, s$ and $C$. 
\end{remark}

\begin{proposition}\label{prop:enumerableSset}
    Let $x\in\R^2$ such that $\dim_H(x) > t > 0$, and let $\epsilon > 0$. Then for every sufficiently large $n$ the following holds:

    If $\P\subseteq \D_{2^{-n}}$ is enumerable given $w$, and $\P$ contains the cube containing $x$, then it has a dyadic $(2^{-n}, t, C)$ subset $\P'$ such that $C= 2^{O(\eps n)+|w|}$ and
    \begin{center}
        $\vert \P^\prime\vert \geq 2^{K_n(x)-O( \eps^{-1}\log n)-|w|}$.
    \end{center}

    In particular, if $|w|=O(\log n)$, then we get $\vert \P^\prime\vert \geq 2^{K_n(x)-O( \eps^{-1}\log n)}$ and $C= 2^{O(\eps n)}$.
\end{proposition}
\begin{proof} 
We define $\P'$ to be the set of those cubes $Q$ in $\P$ for which $$K(Q)\lessapprox K_n(x),\,\text{ and }\,K_{n,m}(Q\mid Q)\lessapprox K_{n,m}(x\mid x)$$ for every $m\le n$ which is an integer multiple of $\lfloor \eps n\rfloor$.

    Note that $\P^\prime$ is enumerable given $w$ and given each $K_{n,m}(x\mid x)$ and $K_{n}(x)$. Moreover, $Q(x) \in \P^\prime$, where $Q(x)$ is the $2^{-n}$-dyadic cube containing $x$. Therefore, 
    \begin{align*}
        K_n(x) &\lessapprox \log \vert \P^\prime \vert + \log K_n(x) + \sum\limits_{m} \log K_{n,m}(x\mid x)+|w|\\
        &\leq \log\vert \P^\prime \vert  + O(\eps^{-1}\log n)+|w|.
    \end{align*}
 
    Hence, \begin{center}
        $\vert \P^\prime\vert \geq 2^{K_n(x)-O(\eps^{-1}\log n)-|w|}$
    \end{center}
    as required. Moreover, since there are at most $2^{K_n(x) + O(\log n)}$ many cubes $Q$ such that $K(Q)\lessapprox K_n(x)$, we see that $$|\P'|\le 2^{K_n(x) + O(\log n)}.$$

    Let $m \leq n$ be an integer multiple of $\lfloor \eps n\rfloor$. Now let $Q\in \D_{2^{-m}}$. By our definition of $\P^\prime$, if $Q'\in\P'$ is a subcube of $Q$, then $K_{n,m}(Q'\mid Q)=K_{n,m}(Q'\mid Q')\lessapprox K_{n,m}(x\mid x)$. For any fixed $Q\in \D_{2^{-m}}$, the number of subcubes $Q^\prime \in \D_{2^{-n}}$ such that $K_{n,m}(Q^\prime \mid Q) \lessapprox K_{n,m}(x\mid x)$ is $\lesssim 2^{K_{n,m}(x\mid x) + O(\log n)}$. Therefore, using symmetry of information, the number of subcubes is at most
     $$2^{K_{n, m}(x\mid x)+O(\log n)}\leq 2^{K_n(x) - K_{m}(x)+O(\log n)}\le \vert \P^\prime\vert 2^{-mt+O(\eps^{-1}\log n)+|w|}.$$

Since this holds for every $m$ which is an integer multiple of $\lfloor \eps n\rfloor$, and since between $m$ and $m+\lfloor \eps n\rfloor$ the number of cubes can be multiplied by at most a factor $2^{O(\eps n)}$, the statement indeed follows with $C=2^{O(\eps n)+O(\eps^{-1}\log n)+|w|}=2^{O(\eps n)+|w|}$.
    
\end{proof}

\begin{theorem}\label{thm:effthm1}
    Let $x$ be a point and $\ell$ be a line in $\R^2$ such that $x\in\ell$, and let $\dim_H(x)>t>0$ and $\dim_H^x(\ell)>s>0$. Let $\epsilon>0$ and let $n\in\N$ be sufficiently large. Then,
    \begin{equation}\label{9}
        K_n(\ell) \ge K_n(\ell\mid x) + \min\{t, (s+t)/2, 1\}n - O(\epsilon n) - O_{\eps,s,t}(\log n),
    \end{equation}
    where the indices $\eps,s,t$ mean that the constant in the $O(\log n)$ term depends on $\eps,s,t$.
\end{theorem}

\begin{proof}Without loss of generality, we can assume that $t$, $s$ and $\eps$ are rational.
 
Define the set
\begin{center}
    $\T = \{T\in\T^{2^{-n}}\mid K(T) \lessapprox K_n(\ell)\},$
    \end{center}
    and denote
$$n_Q = \#\{T\in \T\mid T\cap Q \neq \emptyset\}.$$ 

We note that $\T$ is enumerable given $K_n(\ell)$, and that the $2^{-n}$-tube containing $\ell$ is in $\T$. Therefore, we have
\begin{equation}\label{10}
    \vert \T\vert = 2^{K_n(\ell) + O(\log n)}.
\end{equation}

Also note that, since $\T$ is enumerable given $K_n(\ell)$, for every $Q\in \D_{2^{-n}}$ and for every $T\in\T$ with $Q\cap T\neq \emp$ we have $K_n(T\mid Q)\lessapprox\log n_Q+\log K_n(\ell)$. In particular, $$K_n(\ell\mid x)\lessapprox \log n_{Q(x)}+\log K_n(\ell),$$ where $Q(x)\in\D_{2^{-n}}$ is the dyadic cube containing $x$.

Define the set 
$$\P_1 = \{Q\in \D_{2^{-n}} \mid K(Q) \lessapprox K_n(x), n_Q \geq 2^{K_n(\ell\mid x) - O(\log n)}\}.$$
Then $Q(x)\in\P_1$, and $\P_1$ is enumerable, given $n$, $K_n(\ell)$, $K_n(x)$ and  $K_n(\ell\mid x)$ (all of which take at most $O(\log n)$ bits). 

We fix a rational $\eta_0$ which will be specified later. Define the set $\P$ to be the set of those $Q\in \P_1$ that satisfy the following.  There is a dyadic $(2^{-n}, s, 2^{O(\eta_0 n)})$-set $\T(Q)\subseteq \T$ with $\vert \T(Q)\vert \geq 2^{K_n(\ell\mid x) - O(\eta_0^{-1}\log n)}$ such that $T \cap Q \neq \emptyset$ for all $T\in\mathcal{T}(Q)$. The constants in the $O(\cdot)$ terms will come from Proposition \ref{prop:enumerableSset}.

The set $\T(Q(x)) = \{T\in \T\mid T\cap Q(x) \neq \emptyset\}$ is enumerable given $Q(x)$ and $K_n(\ell)$, and contains the $2^{-n}$-tube containing $\ell$. Therefore, by the dual of Proposition \ref{prop:enumerableSset}, applied to $\T(Q(x))$, we see that $\T(Q(x))$ contains a subset $\T^\prime$ which is a dyadic $(2^{-n}, s, 2^{O(\eta_0 n)})$-set and such that $\vert \T^\prime \vert \geq 2^{K_n(\ell\mid x) - O(\eta_0^{-1} \log n)}$. Therefore we see that $Q(x) \in \P$. Furthermore, $\P$ is computable enumerable, given $s$, $\eta_0$, $K_n(x)$, $K_n(\ell)$ and $K_n(\ell\mid x)$ (recall Remark \ref{remark}, and that we assumed that our parameters are rational). Since of course $K(\eta_0,s)=O(1)$, and the other inputs have length $O(\log n)$, therefore by applying Proposition \ref{prop:enumerableSset} to $\P$, we obtain a set $\P^\prime \subseteq\P$ such that $\vert \P^\prime\vert \geq 2^{K_n(x) - O(\eta_0^{-1} \log n)}$ and $\P^\prime$ is a dyadic $(2^{-n}, t, 2^{O(\eta_0 n)})$-set. 

Now recall Remark \ref{remark0}. To summarize, $\P^\prime$ is a $(2^{-n}, t,C^\prime)$-set, for some fixed constant multiple of $2^{O(\eta_0 n)}$. And for every $Q\in \P^\prime$, there is a $(2^{-n}, s, C^\prime)$-set of tubes $\T(Q)$ such that $\vert \T(Q) \vert = 2^{K_n(\ell\mid x) - O(\eta_0^{-1}\log n)}$. 

Now we choose $\eta_0$ small enough so that we may apply Theorem \ref{thm:RenWang} to $\P^\prime$ with parameters $s$, $t$ and $\eps$ and $\eta:=(\log C')/n\sim\eta_0$. Applying Theorem \ref{thm:RenWang} yields 
\begin{center}
    $\vert \T \vert \gtrsim(\log\eps^{-1})^{-c}2^{\min\{t, \frac{s+t}{2}, 1\}n - \eps n}2^{K_n(\ell\mid x) - O(\eta^{-1}\log n)}$.
\end{center}
Therefore, using also \eqref{10}, we showed that for every $\eps,s,t$ there is an $\eta$ such that
$$K_n(\ell) \gtrapprox \log \vert \T\vert\ge K_n(\ell\mid x) + \min\{t, \frac{s +t}{2}, 1\}n - O(\eps_0 n) - O(\eta^{-1}\log n).$$
\end{proof}

We note that analogs of Proposition \ref{prop:enumerableSset} and Theorem \ref{thm:effthm1} using conditional complexities, instead of oracles, are true as well. 
\begin{proposition}\label{prop:enumerableSset2}
    Let $x\in\R^2$ and $\eps > 0$. Suppose that $n$ is sufficiently large and let $t = \min\{\frac{K_{m}(x)}{m} \mid \epsilon n \leq m \leq n\}$. Then the following holds.

    If $\P\subseteq \D_{2^{-n}}$ is enumerable given $w$, and $\P$ contains the cube containing $x$, then it has a dyadic $(2^{-n}, t, C)$ subset $\P'$ such that $C= 2^{O(\eps n)+|w|}$ and
    \begin{center}
        $\vert \P^\prime\vert \geq 2^{K_n(x)-O( \eps^{-1}\log n) -|w|}$.
    \end{center}
    In particular, if $|w| = O(\log n)$, then $\vert \P^\prime\vert \ge 2^{K_n(x) - O(\eps^{-1} \log n)}$ and $C = 2^{O(\eps n)}$.
\end{proposition}
The proof of this proposition is nearly identical to Proposition \ref{prop:enumerableSset}, and so we omit it. The only difference is that we are only considering the minimum of the ratio $K_m(x)/m$ \textit{up to precision $n$} instead of the limit inferior, i.e., the dimension.

\begin{theorem}\label{thm:effthm2}
    Let $x\in \R^2$, and let $\ell$ be a line in $\R^2$ such that $x\in\ell$. Let $\eps > 0$ and for each $n\in\N$, define $s_n = \min\{\frac{K_{m,n}(e\mid x)}{m} \mid \epsilon n \leq m \leq n\}$. Let $\dim_H(x) > t > 0$ and let $\liminf_n s_n > s > 0$. Let $n$ be sufficiently large. Then,
    \begin{equation}
        K_n(\ell) \geq K_n(\ell \mid x) + \min\{t, \frac{s+t}{2}, 1\}n - O(\epsilon n)-O_{\eps,s,t}(\log n).
    \end{equation}
\end{theorem}
Again, the proof of this is nearly identical to Theorem \ref{thm:effthm1}. 

\section{Proof of main theorem}
We now prove Theorem \ref{thm:mainthm1} using the algorithmic results of the previous section. To do so, we will need to prove the existence of an appropriate oracle $A$ and points satisfying the following properties. For a (suitably chosen) oracle $A$ we need to find $x, y\in\R^2$ and $e = \frac{y-x}{\vert x - y\vert}$ that satisfy:
\begin{itemize}
\item[\textup{(C1)}] $\dim_H^A(x) > \sigma_X, \dim_H^A(y) > \sigma_Y$;
\item[\textup{(C2)}] $\dim_H^{A,x}(e) > 0$;
\item[\textup{(C3)}] $K^{A, x}_n(y) \gtrapprox K^{A}_n(y)$,
\end{itemize}
where the parameters $\sigma_X$ and $\sigma_Y$ we will specify later. 

In \cite{FieStu24} (Lemma 33), Fiedler and Stull used Theorem \ref{cor:OrpShmWan} to give sufficient conditions for the existence of points satisfying the above conditions.
A straight-forward modification of this result proves the following lemma. For completeness, we include the proof.

\begin{lemma}\label{lem:C1-C3}
Let $X, Y\subseteq\mathbb{R}^2$ be compact sets. Let $\sigma_X <\dim_H(X)$ and $\sigma_Y < \dim_H(Y)$. Let $\mu, \nu$ be probability measures whose support is $X$ and $Y$, respectively, and such that $\mu$ and $\nu$ satisfy the $\dim_H(X)$ -dimensional and $\dim_H(Y)$-dimensional Frostman condition, respectively. Let $A$ be an oracle relative to which $X$ and $Y$ are effectively compact and relative to which $\mu$ and $\nu$ are computable.

    Assume $\mu(\ell)\nu(\ell)<1$ for every line $\ell$ and $\nu$ is computable with respect to $A$.  Then, 
\begin{center}
    $\dim_H\{x\in X: \exists y\in Y \text{ satisfying (C1)-(C3)}\}>\sigma_X$.
\end{center}
\end{lemma}
\begin{proof}
Fix $\sigma<\sigma' < \min\{\sigma_X,\sigma_Y, 1\}$. Then $\mu$ and $\nu$ satisfy the $\sigma'$-dimensional Frostman condition. Let $E$ be the Borel set given by Theorem \ref{cor:OrpShmWan}, applied with $\sigma'$. We denote the vertical sections of $E$ by
$$E_x:=\{y\mid (x,y)\in E\}.$$
Since $(\mu \times \nu)(E) >0$, the set
\begin{center}
    $X_1:= \{x\in X\mid \nu(E_x) > 0\}$
\end{center}
has positive $\mu$-measure, and since $\mu$ satisfies the $\dim_H(X)$-dimensional Frostman condition, therefore this implies that the dimension of $X_1$ is $\dim_H(X)$. Since $\sigma_X<\dim_H(X)$, therefore 

\begin{center}
    $X_2:= \{x\in X_1 \mid \dim_H^A(x) > \sigma_X\}$
\end{center}
also has Hausdorff dimension $\dim_H(X)$.

Let $x$ be any point in $X_2$. We know that $\nu$-a.e. $y$ satisfies the condition in (C1), and since $\nu$ is computable with respect to $A$, therefore using \eqref{dd} we also know that $\nu$-a.e. $y$ satisfies (C3). Let
$E^\prime \subseteq E_x$ be the set of points in $E_x$ satisfying both (C1) and (C3). Then we have $\nu(E^\prime) = \nu(E_x) >0$. 

By Theorem \ref{cor:OrpShmWan}, we have $\nu(R\cap E^\prime) \leq C r^{\sigma'}$ for all $r > 0$ and all rectangle $R$ of width $r$ containing $x$. Using very simple geometry, this, combined with Frostman's lemma, implies that $\dim_H(\pi_x E^\prime) \geq \sigma'$. Then for every $\sigma<\sigma''<\sigma'$, for most $e\in \pi_x E'$ (all $e\in \pi_x E'$ except for an at most $\sigma''$-dimensional subset) we have $\dim_H^{A,x}(e) \ge \sigma''>\sigma$.

In particular, there is an $e\in \pi_x E^\prime$ satisfying (C2). Let $y \in E^\prime$ such that $e = \frac{y-x}{\vert x - y\vert}$. Then this choice of $x$ and $y$ satisfies (C1)-(C3).
\end{proof}

\begin{remark}
    Note that the proof of Lemma \ref{lem:C1-C3} gives a stronger result. Not only can we find a points $x$ and $y$ satisfying (C1), (C2) and (C3), but, for every $\sigma < \min\{\sigma_X, \sigma_Y, 1\}$, we can find points satisfying (C1) and (C3) such that 
    \begin{center}
        $\dim^{A,x}_H(e) > \sigma$.
    \end{center}
\end{remark}

\begin{theorem}\label{thm:boundComplexityLine}
    Under the assumptions of Lemma \ref{lem:C1-C3}, there are points $x\in X$ and $y\in Y$ such that $\dim_H^A(x) > \sigma_X, \dim_H^A(y) > \sigma_Y$ and $\dim_H^{A,x}(\ell) > 0$, where $\ell$ is the line joining $x$ and $y$, satisfying the following for all sufficiently large $n$ and $m\leq n$.
\begin{equation*}
K_m^A(\ell)\lessapprox K_{m}^A(\ell\mid x)+K_{m,n}^A(\ell\mid y).
\end{equation*}
Similarly, we have
\begin{equation*}
K_n^{A}(\ell)\lessapprox K_{n}^A(\ell\mid y)+K_{n}^{A,x}(\ell).
\end{equation*}
\end{theorem}
\begin{proof}
Using Lemma \ref{lem:C1-C3}, gives an $x\in X$ and $y\in Y$ satisfying (C1), (C2) and (C3).

  By the symmetry of information, $$K^A_m(\ell)-K^A_{m}(\ell\mid x)\approx K^A_m(x)-K^A_m(x\mid \ell).$$ Applying the symmetry of information again,
$$K^A_m(x\mid \ell)\gtrapprox K^A_{m,m,n}(x\mid \ell,y)\approx K^A_{m,m,n}(x,\ell\mid y)-K_{m,n}(\ell\mid y).$$
Since knowing $y$ with precision $n$ and $x$ with precision $m$ also determines $\ell$ with precision $m$ (so long as $x$ and $y$ are not too close to each other, which we can assume) therefore $$K^A_{m,m,n}(x,\ell\mid y)\approx K^A_{m,n}(x\mid y).$$ Putting these inequalities together we get
$$K^A_m(\ell)- K^A_{m}(\ell\mid x)-K^A_{m,n}(\ell\mid y)\lessapprox K^A_m(x)-K^A_{m,n}(x\mid y).$$ By (C3), $K_m^A(x)\approx K_{m,n}^A(x\mid y)$, and the first equality follows.

Now we repeat essentially the same argument, except that we exchange $x$ and $y$ in the above formulas, put $n=m$, and use $x$ as an oracle.

Since $$K^A_m(\ell)-K^A_m(\ell\mid y)\approx K^A_m(y)-K^A_m(y\mid \ell)\lessapprox K^A_m(y)-K_m^{A,x}(y\mid \ell),$$
and
$$K_m^{A,x}(y\mid \ell)\approx K_{m}^{A,x}(y,\ell)-K_m^{A,x}(\ell)\approx K_{m}^{A,x}(y)-K_m^{A,x}(\ell),$$
we have
$$K^A_m(\ell)-K^A_m(\ell\mid y)-K_m^{A,x}(\ell)\lessapprox K^A_m(y)-K^{A,x}_m(y).$$
By assumption (C3), the right hand side of this is $\approx 0$, so
\begin{equation}
K_m^A(\ell)\lessapprox K_{m}^A(\ell\mid y)+K_{m}^{A,x}(\ell),
\end{equation}
and the proof is complete.
\end{proof}

We are now ready to prove our main theorem.
\begin{proof}[Proof of Theorem \ref{thm:mainthm1}]
    Without loss of generality, we assume that either $X$ is not contained in any line, or that $X \subseteq \ell$, for some line $\ell$, and $Y$ is disjoint from $\ell$.

    Let $\eps>0$, $\sigma_X < \dim_H(X)$, $\sigma_Y < \dim_H(Y)$, and let $X^\prime \subseteq X$, $Y^\prime \subseteq Y$ be compact sets such that $\dim_H(X^\prime) > \sigma_X$ and $\dim_H(Y^\prime) > \sigma_Y$. Let $\mu, \nu$ be probability measures whose support is $X^\prime$ and $Y^\prime$, respectively, and such that they satisfy the $\dim_H(X^\prime)$ and the $\dim_H(Y^\prime)$ dimensional Frostman condition, respectively.
    
    We first handle the case when $\nu(\ell) = \mu(\ell) = 1$ for some line $\ell$. Then $Y$ cannot be disjoint from $\ell$, hence, by our initial assumption, there is a point $x \in X$ such that $x \notin \ell$. The radial projection of $Y$ from this point $x$ has dimension at least $\dim_H(Y^\prime\cap\ell) > \sigma_Y$, and the conclusion follows.

    We now assume that $\mu(\ell)\nu(\ell) < 1$ for every line $\ell$.
    
    Let $A$ be an oracle relative to which $X^\prime$ and $Y^\prime$ are computably compact, and such that $\nu$ and $\mu$ are computable. 

    By Theorem \ref{thm:boundComplexityLine}, there are points $x\in X^\prime$ and $y\in Y^\prime$ satisfying (C1) and (C2), and such that
    \begin{equation}\label{lxy}
K_m^A(\ell)\lessapprox K_{m}^A(\ell\mid x)+K_{m,n}^A(\ell\mid y),
\end{equation}
    and
    \begin{equation}\label{exy2}
K_n^A(\ell)\lessapprox K_{n}^A(\ell\mid y)+K_{n}^{A,x}(\ell),
\end{equation}
for all sufficiently large $n$ and $m$, where $m\leq n$ and $\ell$ is the line joining $x$ and $y$. 

Denoting $t_x = \dim^A_H(x)$ and $s_x = \dim_H^{A,x}(\ell)=\dim_H^{A,x}(e)$, since $s_x > 0$, from \eqref{lxy} and
Theorem \ref{thm:effthm1} we get
\begin{align}
    K^A_{m,n}(\ell\mid y) &\gtrapprox K_m^A(\ell)-K_m^A(\ell\mid x)\tag*{}\\
    &\gtrapprox \min\{t_x, (s_x+t_x)/2, 1\}m - O(\eps m)-O_{\eps,s_x,t_x}(\log m)\label{sy},
\end{align}
for sufficiently large $n$ and $m$ where $m \leq n$. 

Denote $t_y = \dim^A_H(y)$. For every $n$, let $s_n = \min\{\frac{K^A_{m,n}(\ell \mid y)}{m} \mid \epsilon n \leq m \leq n\}$. Let $s_y = \liminf_n s_n$, and note that \eqref{sy} implies that $s_y > 0$. 

We now fix a sufficiently large $n\in\N$ such that $s_n \sim s_y$. Choosing $m \geq \eps n$ to be a precision at which $K^A_{m,n}(\ell\mid y) = K^A_{m,n}(e\mid y) \sim s_y m$, from \eqref{sy} we see that
     \begin{align}
         s_y &\ge\min\{t_x, (s_x+t_x)/2, 1\} -O(\eps)- \frac{O_{\eps,s_x,t_x}(\log n)}{m}\tag*{}\\
         &\ge\min\{t_x, (s_x+t_x)/2, 1\} -O(\eps)\label{eq:1}.
     \end{align} 
     
For any sufficiently large $m\in\N$, by \eqref{exy2} and Theorem \ref{thm:effthm2} we get
     \begin{align}
         K^{A,x}_{m}(\ell) &\gtrapprox K_m^A(\ell)-K_m^A(\ell\mid y)\tag*{}\\
         &\gtrapprox \min\{t_y, (s_y+t_y)/2, 1\}m -O(\eps m)-O_{\eps,s_y,t_y}(\log m)\label{sy2}.
     \end{align}

Similarly, choosing $m$ to be a precision at which $K^{A,x}_m(e) \sim s_x m$, from \eqref{sy2} 
we get
     \begin{align}
         s_x &\ge\min\{t_y, (s_y+t_y)/2, 1\} - O(\eps) - \frac{O_{\eps,s_y,t_y}(\log m)}{m}\tag*{}\\
         &\geq\min\{t_y, (s_y+t_y)/2, 1\} -O(\eps)\label{eq:2} .
     \end{align}

Plugging in  \eqref{eq:1} into \eqref{eq:2} we get:
\begin{equation}\label{15} s_x\ge\min\{t_y,\frac{t_x+t_y}2,\frac{\frac{s_x+t_x}{2}+t_y}2,\frac{1+t_y}2,1\} - O(\eps).
\end{equation}

This implies that
\begin{equation}\label{sx} 
s_x\ge \min\{t_y,\frac{t_x+t_y}{2},1\} -O(\eps).
\end{equation} 
Indeed, in \eqref{15} we have $\frac{1+t_y}{2}\ge\min\{1,t_y\}$, so either \eqref{sx} holds or we have $$s_x\ge \frac{\frac{s_x+t_x}{2}+t_y}{2}-O(\eps).$$ The latter one implies that $$3s_x\ge t_x+2t_y -O(\eps) =2\frac{t_x+t_y}{2}+t_y-O(\eps),$$ so in this case $s_x\ge \min\{\frac{t_x+t_y}{2},t_y\}-O(\eps)$ and therefore \eqref{sx} holds.

From \eqref{sx} the statement of Theorem \ref{thm:mainthm1} follows, since
\begin{align}
   \dim_H^{A,x}(\pi_x y)&=s_x\tag*{}\\
   &\ge \min\{t_y,\frac{t_x+t_y}{2},1\} - O(\eps)\tag*{}\\
   &\ge \min\{\sigma_Y,\frac{\sigma_X+\sigma_Y}{2},1\} - O(\eps).\tag*{}
\end{align}
By \eqref{don}, this implies that $$\dim_H(\pi_x Y)\ge \min\{\sigma_Y, \frac{\sigma_X+\sigma_Y}{2},1\} - O(\eps).$$ 

Since for every $\eps > 0$, and for every $\sigma_X<\dim_H(X)$, $\sigma_Y<\dim_H(Y)$ there is an $x$ for which this holds, the conclusion follows.

\end{proof}

\section{Concluding remarks}
\begin{remark}
    It is interesting to note that the proof of Theorem \ref{thm:mainthm1} does not use the full power of the radial projection result of \cite{OrpShmWan24}. In particular, the proof only uses that the dimension of the radial projection $\pi_x Y$ is strictly positive, for ``many" $x\in X$. This allows us to choose an arbitrary small $\sigma>0$ (other than a $\sigma$ close to $s$) in the theorem below.
\end{remark}

\begin{remark}
It is also possible to prove a discretized analog of Theorem \ref{thm:mainthm1}, using largely the same ideas. The primary difference is that, in the discrete setting, we cannot use Theorem \ref{cor:OrpShmWan}, since its discrete analog is false. This failure is due to the so-called ``train track" counterexample \cite{KatTao01}; we thank Joshua Zahl for bringing this counterexample to our attention.

Nevertheless, in the special case where the conclusion of Theorem \ref{cor:OrpShmWan} does hold, the discrete version remains valid. We can prove the following result.
\end{remark}
\begin{theorem}\label{thm:mainTheoremDiscretized}
    Let $s, t,c,\sigma, \eps > 0$. There exists an $\eta$, depending on $s,t,c, \sigma$ and $\eps$ such that the following holds for every sufficiently small dyadic scale $\delta$. Let $\P_1, \P_2 \subseteq \R^2$ be $(\delta, s, \delta^{-\eta})$- and $(\delta, t, \delta^{-\eta})$-sets, respectively. Let $\mu, \nu \in \mathcal{P}(\R^2)$ be the uniform measures on $\D_\delta(\P_1)$ and $\D_\delta(\P_2)$, respectively. Suppose that there exists a Borel set $E\subset \R^2\times \R^2$ with $(\mu\times\nu)(E) \geq c$ with the following property. For every $x\in \R^2$, $r\geq\delta$ and for every rectangle $x\in R\subset\R^2$ of width $r$,
    \begin{center}
    $\nu(\{y\in R\colon (x,y)\in E\}) \leq \delta^{-\eta} r^\sigma$.\end{center}

    Then there is an $x\in\P_1$ such that $\pi_x \P_2$ contains a $(\delta, \tau, \delta^{-\eps})$-set, for every
    \begin{center}
        $\tau \leq \min\{t, \frac{s+t}{2}, 1\} - O(\eps)$.
    \end{center}
\end{theorem}

\bibliographystyle{amsplain}
\bibliography{pdsc}

@misc{FieStu23,
      title={Dimension of Pinned Distance Sets for Semi-Regular Sets}, 
      author={Jacob B. Fiedler and D. M. Stull},
      year={2023},
      eprint={2309.11701},
      archivePrefix={arXiv},
      primaryClass={math.CA},
      note={arXiv:2309.11701}, 
}

@article{KatTao01,
	author = {Nets Hawk Katz and Terence Tao},
	journal = {New York J. Math},
	pages = {149--187},
	title = {Some connections between {F}alconer's distance set conjecture and sets of {F}urstenburg type},
	volume = {7},
	year = {2001}}

@article{Lutz03a,
	author = {Jack H. Lutz},
	journal = {{SIAM} J. Comput.},
	number = {5},
	pages = {1236--1259},
	title = {Dimension in Complexity Classes},
	volume = {32},
	year = {2003}}

@article{LutLut18,
	address = {New York, NY, USA},
	articleno = {7},
	author = {Lutz, Jack H. and Lutz, Neil},
	issn = {1942-3454},
	issue_date = {May 2018},
	journal = {ACM Trans. Comput. Theory},
	month = may,
	number = {2},
	pages = {7:1--7:22},
	publisher = {ACM},
	title = {Algorithmic Information, Plane {K}akeya Sets, and Conditional Dimension},
	volume = {10},
	year = {2018}}

@article{LutStu20,
	author = {Lutz, Neil and Stull, D. M.},
	journal = {Information and Computation},
	pages = {104601},
	publisher = {Elsevier},
	title = {Bounding the dimension of points on a line},
	volume = {275},
	year = {2020}}

@article{Mayordomo02,
	author = {Elvira Mayordomo},
	journal = {Inf. Process. Lett.},
	number = {1},
	pages = {1--3},
	title = {A {K}olmogorov complexity characterization of constructive {H}ausdorff dimension},
	volume = {84},
	year = {2002}}

@Article{OrpShmWan24,
  author    = {Orponen, Tuomas and Shmerkin, Pablo and Wang, Hong},
  journal   = {Geometric and Functional Analysis},
  title     = {Kaufman and Falconer estimates for radial projections and a continuum version of Beck’s theorem},
  year      = {2024},
  number    = {1},
  pages     = {164--201},
  volume    = {34},
  publisher = {Springer},
}

@Article{ShmWan25,
  author    = {Shmerkin, Pablo and Wang, Hong},
  journal   = {Geometric and Functional Analysis},
  title     = {On the Distance Sets Spanned by Sets of Dimension d/2 in},
  year      = {2025},
  number    = {1},
  pages     = {283--358},
  volume    = {35},
  publisher = {Springer},
}

@article{KelShm19,
	author = {Keleti, Tam\'{a}s and Shmerkin, Pablo},
	doi = {10.1007/s00039-019-00500-9},
	issn = {1016-443X},
	journal = {Geometric and Functional Analysis},
	mrnumber = {4034924},
	number = {6},
	pages = {1886--1948},
	title = {New bounds on the dimensions of planar distance sets},
	volume = {29},
	year = {2019}}

@article{Orponen19DimSmooth,
	author = {Orponen, Tuomas},
	doi = {10.2140/apde.2019.12.1273},
	issn = {2157-5045},
	journal = {Analysis \& PDE},
	mrnumber = {3892404},
	number = {5},
	pages = {1273--1294},
	title = {On the dimension and smoothness of radial projections},
	volume = {12},
	year = {2019}}

@article{Shmerkin20,
  title={Improved bounds for the dimensions of planar distance sets},
  author={Shmerkin, Pablo},
  journal={Journal of Fractal Geometry},
  volume={8},
  number={1},
  pages={27--51},
  year={2020}
}

@article{RenWan23,
  author        = {Kevin Ren and Hong Wang},
journal = {arXiv preprint arXiv:2308.08819},
  month         = aug,
  title         = {Furstenberg sets estimate in the plane},
  year          = {2023},
  archiveprefix = {arXiv},
  copyright     = {Creative Commons Attribution 4.0 International},
  doi           = {10.48550/ARXIV.2308.08819},
  eprint        = {2308.08819},
  primaryclass  = {math.CA},
  publisher     = {arXiv},
  url           = {https://arxiv.org/abs/2308.08819},
}

@misc{FieStu24,
      title={Pinned Distances of Planar Sets with Low Dimension}, 
      author={Jacob B. Fiedler and D. M. Stull},
      year={2024},
      eprint={2408.00889},
      archivePrefix={arXiv},
      primaryClass={math.CA},
      note={arXiv:2408.00889}, 
}

@article{Ren23,
    title={Discretized Radial Projection in {$\mathbb{R^d}$}},
    author={Ren, Kevin},
journal={arXiv preprint arXiv:2309.04097},
    year={2023},
    eprint={2309.04097},
    archivePrefix={arxiv},
    primaryClass={math.CA}
    }

@article{WanZah25,
  title={Volume estimates for unions of convex sets, and the Kakeya set conjecture in three dimensions},
  author={Wang, Hong and Zahl, Joshua},
  journal={arXiv preprint arXiv:2502.17655},
  year={2025}
}

@misc{DuOuRenZhang23a,
    title={New improvement to Falconer distance set problem in higher dimensions},
    author={Du, Xiumin and Ou, Yumeng and Ren, Kevin and Zhang, Ruixiang},
    year={2023},
    eprint={2309.04103},
    archivePrefix={arxiv},
note={arXiv:2309.04103},
    primaryClass={math.CA}}

@article{Hitchcock05,
	author = {Hitchcock, John M.},
	journal = {Theory of Computing Systems},
	number = {5},
	pages = {559--571},
	title = {{Correspondence principles for effective dimensions}},
	volume = {38},
	year = {2005}}

@Misc{Stull22c,
  author        = {D. M. Stull},
  note          = {arXiv:2207.12501},
  title         = {Pinned Distance Sets Using Effective Dimension},
  year          = {2022},
  archiveprefix = {arXiv},
  eprint        = {2207.12501},
  primaryclass  = {cs.CC},
}

@Article{BriFuRen24,
  author  = {Bright, Paige and Fu, Yuqiu and Ren, Kevin},
  journal = {arXiv preprint arXiv:2406.09707},
  title   = {Radial Projections in $\mathbb{R}^n$  Revisited},
  year    = {2024},
}

@article{Lutz21,
	author = {Neil Lutz},
	title = {Fractal Intersections and Products via Algorithmic Dimension},
	journal = {ACM Transactions on Computation Theory},
    month = aug,
    articleno = {14},
    numpages = {15},
    year = {2021},
    issue_date = {September 2021},
    publisher = {Association for Computing Machinery},
    address = {New York, NY, USA},
    volume = {13},
    number = {3},
    doi = {10.1145/3460948},
}

@Article{AltBusWil25,
  author    = {Altaf, Iqra and Bushling, Ryan and Wilson, Bobby},
  journal   = {Real Analysis Exchange},
  title     = {Distance Sets Bounds for Polyhedral Norms via Effective Dimension},
  year      = {2025},
  number    = {1},
  volume    = {1},
  publisher = {Michigan State University Press},
}

@article{LuQiYu24,
    author = {Jack H. Lutz and Renrui Qi and Liang Yu},
    title ={The Point-to-Set Principle and the dimensions of {H}amel bases},
    journal = {Computability},
    volume = {13},
    number = {2},
    pages = {105-112},
    year = {2024},
    doi = {10.3233/COM-210383},
}

@Article{BusFie25,
  author  = {Bushling, Ryan EG and Fiedler, Jacob B},
  journal = {Journal of Fractal Geometry},
  title   = {Bounds on the dimension of lineal extensions},
  year    = {2025},
  number  = {1},
  pages   = {105--133},
  volume  = {12},
}

@article{LuLuMa23a,
    title = {Extending the reach of the point-to-set principle},
    journal = {Information and Computation},
    volume = {294},
    pages = {105078},
    year = {2023},
    issn = {0890-5401},
    doi = {https://doi.org/10.1016/j.ic.2023.105078},
    author = {Jack H. Lutz and Neil Lutz and Elvira Mayordomo},
}

@article{OrpShm23,
  title={Projections, Furstenberg sets, and the $ ABC $ sum-product problem},
  author={Orponen, Tuomas and Shmerkin, Pablo},
  journal={arXiv preprint arXiv:2301.10199},
  year={2023}
}

@article{WanZah22,
  title={Sticky Kakeya sets and the sticky Kakeya conjecture},
  author={Wang, Hong and Zahl, Joshua},
  journal={arXiv preprint arXiv:2210.09581},
  year={2022}
}

\end{document}